\documentclass[12pt,draft,reqno]{amsart}

\usepackage[all]{xy}
\usepackage{amsthm,array,amssymb,amscd,amsfonts,latexsym}

\usepackage{enumerate}

\usepackage{paralist}
\usepackage{mathrsfs}
\usepackage{mathrsfs}
\usepackage{dsfont}
\usepackage{bbold}
\usepackage{mathbbol}
\usepackage[all]{xy}
\usepackage{enumerate}

\usepackage{mathbbol}

{\catcode`\@=11
\gdef\n@te#1#2{\leavevmode\vadjust{%
 {\setbox\z@\hbox to\z@{\strut#1}%
  \setbox\z@\hbox{\raise\dp\strutbox\box\z@}\ht\z@=\z@\dp\z@=\z@%
  #2\box\z@}}}
\gdef\leftnote#1{\n@te{\hss#1\quad}{}}
\gdef\rightnote#1{\n@te{\quad\kern-\leftskip#1\hss}{\moveright\hsize}}
\gdef\?{\FN@\qumark}
\gdef\qumark{\ifx\next"\DN@"##1"{\leftnote{\rm##1}}\else
 \DN@{\leftnote{\rm??}}\fi{\rm??}\next@}}

\DeclareFontFamily{OT1}{wncyr}{\hyphenchar\font45
}
\DeclareFontShape{OT1}{wncyr}{m}{n}{%
   <5> <6> <7> <8> <9> gen * wncyr
   <10> <10.95> <12> <14.4> <17.28> <20.74>  <24.88>wncyr10}{}
\DeclareFontShape{OT1}{wncyr}{m}{it}{%
   <5> <6> <7> <8> <9> gen * wncyi
   <10> <10.95> <12> <14.4> <17.28> <20.74> <24.88> wncyi10}{}
\DeclareFontShape{OT1}{wncyr}{m}{sc}{%
   <5> <6> <7> <8> <9> <10> <10.95> <12> <14.4>
   <17.28> <20.74> <24.88>wncysc10}{}
\DeclareFontShape{OT1}{wncyr}{b}{n}{%
   <5> <6> <7> <8> <9> gen * wncyb
   <10> <10.95> <12> <14.4> <17.28> <20.74> <24.88>wncyb10}{}
\input cyracc.def

\DeclareMathSizes{9}{9}{7}{5}

\theoremstyle{plain}

\newtheorem{theorem}{Theorem}
\newtheorem{proposition}[theorem]{Proposition}
\newtheorem{lemma}[theorem]{Lemma}
\newtheorem{corollary}[theorem]{Corollary}

\theoremstyle{definition}

\newtheorem{nothing*}[theorem]{}
\newtheorem{subnothing*}[sub]{}
\newtheorem{example}[theorem]{Example}

\theoremstyle{remark}

\newtheorem{remark}[theorem]{\bf Remark}

\newtheorem{subsec}{\bf \hskip -.5mm}

\newcommand{\cc}{\raise .4pt \hbox{{$\scriptstyle{\bullet}$}}}

\begin{document}

\title[Embeddings of groups
${\rm Aut}(F_n)$]{
Embeddings of groups \boldmath${\rm Aut}(F_n)$ \\
into automorphism groups\\
of algebraic varieties
}
\author[Vladimir L. Popov]{Vladimir L. Popov}
\address{Steklov Mathematical Institute,
Russian Academy of Sciences, Gub\-kina 8,
Moscow 119991, Russia}
\email{popovvl@mi-ras.ru}

\dedicatory{To the memory of J.\;E.\;Humphreys}

\maketitle

{\def\thefootnote{\relax}

\begin{abstract}
For every positive integer $n$, we construct, using algeb\-raic groups, an infinite family of irre\-du\-c\-ible
algebraic varieties $X$,
whose automorphism group ${\rm Aut}(X)$ contains
 the automorphism group ${\rm Aut}(F_n)$ of a free group $F_n$ of rank
 $n$ as a subgroup. This property implies that, for $n \geqslant 2$, such groups ${\rm Aut}(X)$ are non\-amenable, and, for $n \geqslant 3$, 
 nonlinear and contain the braid group $B_n$ on $n$ strands.
 Some of these varieties $X$ are
 affine, and among affine, some are rational and some are not, some are smooth and some are singular.
 As an application, we deduce
 that, for $n \geqslant 3 $, every Cremona group
of rank $\geqslant 3n$ contains  the groups ${\rm Aut}(F_n)$ and
$B_n$ as the subgroups. This bound is better than the one that follows from the paper by D. Krammer \cite{5-12a}, where the linearity of the braid group
$B_n$ is proved.
 \end{abstract}


\begin{subsec}{\bf Introduction.}
In the last decade,
studying biregular and birational automorphism groups of algebraic varieties, in particular, their abstract group pro\-perties and subgroups, has become a trend. In terms of popu\-larity,
the Cremona groups are probably the leaders among the groups in the focus of these explorations.

In this paper, for every positive integer $n$, we construct, using algeb\-raic groups,  an infinite family of irreducible algebraic varieties $X$, on which the automorphism group ${\rm Aut}(F_n)$ of a free group $F_n$ of rank $n$ acts (algebraically).
We explore the properties of these actions. Central to us is the question of the faithfulness of these actions,
because every such faithful action
defines an embedding of the group ${\rm Aut}(F_n)$ into
the automorphism group ${\rm Aut}(X)$ of the variety $X$, and therefore, provides an information about the subgroups and group properties of ${\rm Aut}(X)$.
We find an infinite family of varieties $X$, for which the constructed action is faith\-ful.


  This is applied to the linearity problem for automorphism groups of algebraic varieties considered in \cite[Prop. 5.1]{14-8},
  \cite{3-9}, \cite{15-6}: the exis\-tence in ${\rm Aut}(X)$ of a subgroup isomorphic to ${\rm Aut}(F_n)$ implies that
  the group ${\rm Aut}(X)$
  is nonamenable for $ n \geqslant 2$ and nonlinear for $n\geqslant 3$. In addition, since some important groups (for example, the braid group $B_n$ on $n$ strands for $n\geqslant 3$) are the
  subgroups of ${\rm Aut}(F_n)$,
  we obtain their
  realization in the form of subgroups
  of automorphism groups of the constructed algebraic varieties.

 Among these varieties some are affine (some of them smooth and some with singularities), and among them some are rational and some are nonra\-tio\-nal (and even not stably rational).

As another application we obtain that
for $n\geqslant 3$, every Cremona group
of rank $\geqslant 3n$ contains the group ${\rm Aut}(F_n)$
and the braid group\;$B_n$ as the subgroups. This bound for the rank of the Cremona group containing
$ B_n $ is better than that following from \cite[Thm. 4.6] {5-12a}, where the linearity of the group $ B_n $ is proved.


If instead of algebraic groups in the considered construction
we take real Lie groups,
this yields results on topological transfor\-ma\-tion groups. In particular,
in this way we obtain families of differenti\-able manifolds (among which there are also compact ones), whose dif\-feo\-mor\-phism groups contain ${\rm Aut}(F_n)$.

\vskip 2mm

{\it Acknowledgement.}   The author is grateful to N.\;L.\;Gordeev for
dis\-cus\-sion of problems about group identities
arising from the proof of Theo\-rem \ref{fg},
information about some publications on this topic, and
com\-ments on the first version of this paper.

\end{subsec}

\begin{subsec}{\bf Conventions and notation.}  In what follows, algebraic varieties are considered over an algebrai\-cally closed field $k$. With regard to algebraic geometry and algebraic groups, we
follow conventions and notation from \cite{9-2}
used freely.


${\rm Sym}(X)$ denotes the symmetric group of a set $X$.

If $X$ is an algebraic variety (respectively, a differentiable manifold), then the subgroup of
${\rm Sym}(X)$ consisting of all automorphisms (respecti\-vely, diffeomorphisms) of $X$ is denoted by ${\rm Aut}(X)$.

Groups are considered in multiplicative notation.

The statement that a group $G$ contains a group $H$
means the exis\-tence of a monomorphism $\iota \colon H \hookrightarrow G$,
by means of which $H$ is identified with $ \iota(H) $.


${\mathscr C}(G)$ denotes the  center of a group $G$.

${\rm Lie}(G)$ denotes the Lie algebra of an algebraic group $G$.

$\underline{G}$ denotes the underlying variety (resp. manifold) of an algebraic group (resp. real Lie group) $G$.

${\rm Aut}(G)$ and ${\rm End}(G)$ denote, respectively, the automorphism group and the endomorphism monoid of a group $G$. If $G$ is an algebraic group or a real Lie group,
then by automorphisms we mean automorphisms in the category of algebraic or real Lie groups, so ${\rm Aut}(G)$ is the in\-ter\-section of ${\rm Aut}({\underline G})$ with the automorphism group of the abstract group\;$G$.


\end{subsec}

\begin{subsec}{\bf Word-defined mappings and their properties.}
Consider a posi\-tive integer $n$,
an abstract group $G$ with the identity element $e$, and the group
\begin{equation}\label{G}
X:=G^{n}:=G\times\cdots\times G\quad \mbox{($n$ factors).}
\end{equation}

Let $F_n$ be a free group
of rank $n$ with a free system of generators $ f_1, \ldots,  f_n$.
We denote the identity element in $ {\rm Aut}(F_n) $ by $1$.

For any elements $w\in F_n$ and
\begin{equation}\label{gh}
x=(g_1, \ldots, g_n)\in X,\;\;g_j\in G,
\end{equation}
denote by
$$w(x)=w(g_1,\ldots, g_n)$$
the element of $G$ that is the image of element $w$
under the homomor\-phism $F_n\to G$ that maps $f_j$ to $g_j$  for each $j$.
In other words, $w(x)$ is obtained from the word $w$ in $f_1, \ldots, f_n$
by substituting
$g_j$ in place  of $ f_j$ for each $ j $.

Every element $\sigma \in {\rm End}(F_n)$ is uniquely determined by
the sequence of elements $\sigma(f_1), \ldots, \sigma (f_n) $, and any sequence of elements
of $F_n$ of length $n$ is of this form for some $\sigma $. This sequence
defines the following map (which, in general,
is not an endomorphism of the group $ X $):
\begin{equation}\label{mor}
\sigma_X\colon X\to X,\quad x\mapsto (\sigma(f_1)(x),\ldots, \sigma(f_n)(x)).
\end{equation}

Some properties of such maps are listed below.


\begin
{proposition}\label{pro}
\label{prop} Let $\sigma$ and $\tau$ be any elements of ${\rm End}(F_n)$. Then
the following hold:
\begin{enumerate}[\hskip 2.2mm\rm(a)]
\item ${(\sigma\circ \tau)_{X}^{\ }}=\tau_{X}\circ \sigma_{X}^{\ }$.
    \item
    $1_X={\rm id}$.
    \item
$\sigma_{X}(S^{n})
\subseteq S^{n}$ for every subgroup $S$ of the group $G$.
\item The following properties of element {\rm \eqref{gh}} are equivalent:
\begin{enumerate}[\hskip 0mm\rm(a)]
\item $\sigma_X(x)=x$ for every $\sigma\in {\rm Aut}(F_n)$;
\item if $n>1$, then $g_1=\ldots=g_n=e$, and if $n=1$, then $g_1^2=e$.
\end{enumerate}
\item Let $\gamma\colon G\to H$ be a group homomorphism,
and $Y:=H^{n}$. Then the map
\begin{align*}
{\gamma}_n\colon X
\to Y, \quad
(g_1,\ldots, g_n)
\mapsto (\gamma(g_1),\ldots, \gamma(g_n))
\end{align*}
is ${\rm End}(F_n)$-equivariant, i.e.,
${\gamma}_n\circ \sigma_{X}
=\sigma_{Y}\circ {\gamma}_n.
$
\item
    $\sigma_{X}(xz)=\sigma_{X}(x)\sigma_{X}(z)$ for all $x\in
X$, $z\in {\mathscr C}(X)$. In particular, the restriction of
$\sigma_{X}$
to the group ${\mathscr C}(X)$
is its endomorphism.
\item
$\sigma_X$ commutes with the diagonal
action of $G$ on $X$ by conjugation.
\item If $G$ is an algebraic group
 \textup(respectively, a real Lie group\textup), then
$\sigma_X$ is a morphism \textup(respectively, a differentiable mapping).
\end{enumerate}
\end{proposition}

\begin{proof} Statement (d) follows from the fact that if $n = 1$, then the only element of the group $ {\rm Aut}(F_n)$ not equal to $1$ maps $ f_1 $ to $ f_1^{- 1} $, and
if $ n \geqslant 2 $, then for any $ i, j \in \{1, \ldots, n \} $, $ i \neq j $,
the element $ \sigma_{ij} \in {\rm End} (F_n) $ defined by formula
\begin{equation*}
\sigma_{ij}(f_{l})=\begin{cases}
f_l&\mbox{if $l\neq i$},\\
f_if_j&\mbox{if $l=i$},
\end{cases}
\end{equation*}
 lies in ${\rm Aut}(F_n)$.

The rest of the statements follow directly from the definitions and the fact
that each element of the group $ F_n $ is written as
a non-commutative Laurent monomial in $ f_1, \ldots, f_n $, i.e.,
\begin {equation} \label{aword}
f_ {i_1}^{\varepsilon_1} \cdots f_{i_s}^{\varepsilon_s}, \quad \mbox
{where $ \varepsilon_j\in \mathbb Z $}
\end {equation}
  (uniquely defined if monomial \eqref{aword} is reduced).
\end{proof}

\end{subsec}

\begin{subsec}{\bf On the kernel of the constructed action of
\boldmath${\rm Aut}(F_n)$. } It follows from statements (a), (b), (g) of Proposition  \ref{pro}
that if $\sigma \in {\rm Aut}(F_n) $, then
$\sigma_ {X}\in {\rm Sym}(X)$
with $(\sigma_X)^{-1}=(\sigma^{-1})_X$, and the map
\begin{equation}\label{action2}
{\rm Aut}(F_n)\to {\rm Sym}({X}),\quad \sigma\mapsto(\sigma^{-1})_X,
\end{equation}
is a group homomorphism.
Moreover, if $ G $ is an algebraic group or a real Lie group, then
$ \sigma_X \in {\rm Aut} ({\underline X}) $ and we obtain a group homomorphism

\begin{equation}\label{restri}
{\rm Aut}(F_n)\to {\rm Aut}({\underline X}),\quad \sigma\mapsto(\sigma^{-1})_X.
\end{equation}


\begin{theorem}\label{theo} \

\begin{enumerate}[\hskip 4.2mm\rm(a)]
\item If the group $G$ is solvable and $n\geqslant 3$, then homomorphism
 {\rm \eqref{action2}} is not an embedding.
    \item If the group $G$ is nonsolvable, then homomorphism {\rm \eqref{restri}} is an embedding provided that
     any of the following two properties hold:
     \begin{enumerate}[\hskip 0mm\rm(i)]
     \item[$({\rm b}_1)$] $G$ is a connected \textup(not necessarily affine\textup) algebraic group;
     \item[$({\rm b}_2)$] $G$ is a connected real Lie group.
     \end{enumerate}
\end{enumerate}
\end{theorem}

\begin{proof} (a) Let  $n\geqslant 3$ and let the group $G$ be solvable,
i.e., the
descend\-ing
series of its successive commutator subgroups $G={\mathscr D}^0(G)\supseteq \cdots \supseteq
{\mathscr D}^i(G)\supseteq \cdots$ terminates  at
an $s$th step:
\begin{equation}\label{comm}
{\mathscr D}^s(G)\!=\!\{e\}.
\end{equation}

Setting $[a,b]:=aba^{-1}b^{-1}$, we define inductively the elements $\theta_i\in F_n$
as follows:
\begin{equation}\label{theta}
\begin{split}
\theta_1&:=[f_1, f_2],\\[-3.5mm]
\theta_i&:=[\theta_{i-1}, f_d], \;\;\mbox{where}\;\;d=\begin{cases}1&\text{if $i$ is even},\\
2&\text{if $i$ is odd}, \end{cases}
\;\;\mbox{for $i\geqslant 2$}.
\end{split}
\end{equation}

From \eqref{theta} it follows that $\theta_i$ for any $i$ is a nonempty
reduced word in $f_1$ and $f_2$.
Consider the elements $\sigma, \tau\in {\rm End}(F_n)$
given by the equalities
\begin{equation}\label{sigmatau}
\sigma(f_i):=
\begin{cases}
f_i&\text{if $i<n$,}\\
f_n\theta_s& \text{if $i=n$,}
\end{cases}\quad
\tau(f_i):=
\begin{cases}
f_i&\text{if $i<n$,}\\
f_n\theta_s^{-1}& \text{if $i=n$.}
\end{cases}
\end{equation}

Since $ n \geqslant 3 $, and $ \theta_s $ is a word only in $ f_1 $, $ f_2 $, it follows
from \eqref{sigmatau} that $ \sigma \circ \tau = \tau \circ \sigma = 1$.
Hence
$ \sigma $ is a nonidentity element of the group $ {\rm Aut}(F_n) $ with $ \sigma^{-1}
= \tau $.

Consider now an element \eqref{gh} from $ X $. In view of
\eqref {sigmatau},
  we have
\begin{equation}\label{fig}
\sigma(f_i)(x)=f_i(x)=g_i\quad\mbox{for $i<n$}
\end{equation}
and $\sigma(f_n)(x)=(f_n\theta_s)(x)$. It follows from \eqref{theta} that $\theta_s(x)\in \mathscr D^s(G)$. In view of \eqref{comm}, this yields $\theta_s(x)=e$. Therefore,
\begin{equation}\label{nn}
\sigma(f_n)(x)=f_n(x)=g_n.
\end{equation}
Thus, in view of  \eqref{gh}, \eqref{mor}, \eqref{fig},
\eqref{nn}, we have $ \sigma_X = {\rm id} $. Hence, $ \sigma $ is a non\-iden\-tity element of the kernel of  homomorphism \eqref{restri}. This proves statement\;(a).

(b)  Let $G$ be a nonsolvable group. Arguing by contradiction, suppose that homomorphism \eqref{restri} is not an embedding. Let
$ \sigma \in {\rm Aut}(F_n) $ be a nonidentity element of its kernel.
Since $\sigma\neq 1$, there is a number $i$ such that the reduced word
$\sigma(f_i)$ in $f_1,\ldots, f_n$ is different from $f_i$.
Therefore,
\begin{equation}\label{non}
w:=\sigma(f_i)f_i^{-1}
\end{equation}
is a nonidentity element of the group $F_n$.
According to the choice of $\sigma$,
  for every $x=(g_1,\ldots, g_n)\in G^{n}$, we have $\sigma_X(x)=x$, so \eqref{mor} gives $\sigma(f_i)(x) =g_i=f_i(x)$. Therefore,
\begin{equation}\label{we}
w(x)=e\;\;\mbox{for every $x\in G^{n}$}.
\end{equation}

{\it Case} $({\rm b}_1)$.  Let $G$ be a nonsolvable connected algebraic group. By Chevalley's theorem, the group $G$ contains
a largest connected affine normal subgroup
 $G_{\rm aff}$, and the group $G/G_{\rm aff}$ is an abelian variety.
 Since $G$ is nonsol\-v\-able and  $G/G_{\rm aff}$ is abelian (hence solvable), the group
$G_{\rm aff}$ is nonsolv\-able. Therefore, the group $G_{\rm aff}$ does not coincide with its
radical ${\rm Rad}(G_{\rm aff})$ and hence
$G_{\rm aff}/{\rm Rad}(G_{\rm aff})$ is a nontrivial connected semi\-simple algebraic group. In view statements (c) and (e) of Propo\-sition \ref{pro}, this reduces the proof  to the case where $G$ in $({\rm b}_1)$ is a nontrivial connected semisimple algebraic group.
We will therefore further assume that
this condition is met.

If $n=1$, then ${\rm Aut}(F_n)$ is the group of order two and
 $\sigma(f_1)=f_1^{-1}$, hence $g=\sigma_X(g)=g^{-1}$ for every $g\in G$. Therefore, the group $G$ is commutative,  which contradicts its unsolvability.

So $n\geqslant 2$. Then it follows from \cite[Thm.\,B]{2} and \eqref{non} that the morphism
\begin{equation*}\label{boldw}
{G}^{n}\to {G},\quad x\mapsto w(x)
\end{equation*}
is dominant. In view of \eqref{we}, this means that the group $G$ is trivial.
The resulting contradiction completes the proof in the case of $({\rm b}_1)$.

{\it Case} $({\rm b}_2)$. Let $G$ be a nonsolvable connected real Lie group. Then, according to \cite[Thm.]{12-11}, the group $G$ contains a free subgroup $S$ of rank $n$.
Let $g_1, \ldots, g_n$  be a free generating system of the group $S$. Then it follows from
\eqref{we} that $w(g_1,\ldots, g_n)=e$, i.e., $w$ is a nontrivial relation between
$g_1,\ldots, g_n$, which is impossible.
The resulting contradiction completes the proof in the case of
  $({\rm b}_2)$.

  This proves state\-ment\;(b).
\end{proof}

\begin{remark} If the field $k$ is uncountable, then, according to \cite[Thm. 1.1]{10-4}, \cite[App. D]{13-5}, in the case of $({\rm b}_1)$,
the group $G$ contains a free subgroup of any finite rank, hence
 the same proof as in the case of $({\rm b}_2)$ above is applicable. It
was given in the first version of this paper \cite{22}. However, in the general case, such a proof fails, since $G$ may not contain a free subgroup (for example, if $k$ is the algebraic closure of a finite field, then the order of every element in ${\rm SL}_d$
is finite). The author is grateful to N. L. Gordeev, who drew his attention to Borel's theorem \cite[Thm.\,B]{2}, which allows us, by changing the proof, to remove
the assumption that the field\;$k$ is uncountable.
\end{remark}

\end{subsec}


\begin{subsec}{\bf Corollaries and examples.}   Corollaries \ref{cor1} and \ref{Cr} formulated below
 follow from Theorem \ref{theo} in view of the
 next Proposition \ref{linea}:

\begin{proposition}\label{linea} Let $H$ be a group containing ${\rm Aut}(F_n)$. Then
the follow\-ing hold:
\begin{enumerate}[\hskip 4.2mm\rm(a)]
\item The group $H$ contains the group ${\rm Aut}(F_s)$ for every $s\leqslant n$.
\item If $n\geqslant 2$, then the group $H$ contains the group $F_m$ for every $m$ and is nonamen\-able.
\item If $n\geqslant 3$, then the group  $H$ is nonlinear.
\item If $n\geqslant 3$, then the group $H$ contains the braid group $B_n$ on $n$ strands.
\end{enumerate}
\end{proposition}

\begin{proof}
Clearly the group ${\rm Aut}(F_n)$ contains ${\rm Aut}(F_s)$ for every $s\!\leqslant\! n$.
This gives\;(a).

If $n\geqslant 2$, then the group ${\mathscr C}(F_n)$  is trivial (see \cite[Chap.\,I, Prop. 2.19]{24-13}). Therefore, the group ${\rm Int}(F_n)$
is isomorphic to $F_n$, and hence  (see \cite[Chap.\,I, Prop. 3.1]{24-13})
contains $F_m$ for every $m$.
This gives (b).

If $n\!\geqslant\! 3$, then the group ${\rm Aut}(F_n)$ is nonlinear  (see \cite{5-12}). This gives\;(c).

If $n\!\geqslant\! 3$, then the group ${\rm Aut}(F_n)$ contains  $B_n$ (see \cite[Chap.\,3, 3.7]{18-14}). This gives\;(d).
\end{proof}

\begin{corollary}\label{cor1}
Let
$G$ be
 either a connected algeb\-raic group
or a con\-nec\-ted real Lie group, and let $G$ be nonsolvable. Then for any positive integers $n$ and $s\leqslant n $,
the following hold:
\begin{enumerate}[\hskip 4.2mm $\bullet$]
\item the group ${\rm Aut}(\underline{G^{n}})$
contains the group ${\rm Aut}(F_s)$;
\item for $n\geqslant 2$, the group ${\rm Aut}(\underline{G^{n}})$ contains the braid group $B_n$;
\item   for  $n\geqslant 3$, the group ${\rm Aut}(\underline{G^{n}})$ is nonlinear;
\item  for  $n\geqslant 2$, the group ${\rm Aut}(\underline{G^{n}})$ is nonamenable.
    \end{enumerate}
\end{corollary}

\begin{remark} The action of $ G $ on itself by left or right translations
defines an embedding of $ G $ into $ {\rm Sym}(G) $. Therefore,
  $ {\rm Sym} (G^{n}) $ contains $ G^{n} $.
If $ G $ is a connected algebraic group
or a real Lie group, then this construction shows that
  $ {\rm Aut} (\underline {G^{n}}) $ contains the group $ G^{n} $, which acts on
  $ \underline {G^{n}} $ simply transitively. If, moreover, $ G $ is nonsolvable,
  then, as noted in the proof of Theorem \ref{theo}, it contains $ F_m $ for any $ m $ and, therefore, is nonamenable.
  Hence the same is true for the group
  $ {\rm Aut} (\underline {G^{n}}) $.
\end{remark}

\begin{example}
If we take $ G = {\rm SL}_2(k) $,
then
$ \underline {G^{n}} $ is isomorphic to $ Q^n:=Q \times \cdots \times Q $ ($ n $ factors), where
$ Q $ is the affine quadric in $ \mathbb A^4 $ defined by the equation $ x_1x_2 + x_3x_4 = 1 $. Thus, according to Corollary \ref{cor1},
the automorphism group of the irreducible smooth affine
algebraic variety $Q^n$
contains $ {\rm Aut}(F_s) $ for any $ s \leqslant n $, and if
 $ n \geqslant 3 $, then it is nonlinear and contains the braid group $ B_n $.
The same is true if, for $ {\rm char} (k) \neq 2 $, one replaces $ Q $ with $ Q / Z $, where $ Z $ is the  group of order two generated by the involution $ ( a, b, c, d) \mapsto (-a, -b, -c, -d) $: this follows from Corol\-lary \ref{cor1} for $ G = {\rm PSL}_2 (k) $.
\end{example}

\begin{example} The following construction \cite[Sect. 2, Example 4] {7-17} gives an
embedding of the group $ {\rm Aut}(F_n) $ into the automorphism group  of an affine open subset of the affine space $ \mathbb A^m $ for some $ m $, and hence into the Cremona group
$ {\rm Cr}_m$ of rank $ m $.
In principle, it allows one to describe the birational transformations
of the space $ \mathbb A^m $ that lie in the image of ${\rm Aut}(F_n) $ by means of  the explicit formulas.

Namely, consider a $d $-dimensional associative $ k $-algebra $ A $ with an identity element.
Having fixed its basis, we identify the set $ A $ with $ \mathbb A^d $. The group $ A^* $ of invertible elements of the algebra $ A $ is a connected affine algebraic group whose underlying variety is an  affine open subset of $ \mathbb A^d $. If $ A^* $ is nonsolvable, then formulas \eqref{mor}, \eqref{restri} for $ G = A^* $ define an embedding of
$ {\rm Aut}(F_n) $ into the Cremona group of rank $ nd $. For example, this is so for
$$ A = {\rm Mat}_{s \times s} (k), \quad s \geqslant 2.$$
In this case, $ A^* $ is the nonsolvable group $ {\rm GL}_s $, and explicit descrip\-tion of birational transformations $ \sigma_X $ in the coordinates is reduced to
multiplying functional matrices and their inverses.
\end{example}

\begin{corollary}\label{Cr}  For any integers $ n\geqslant 1 $,
$ r \geqslant 3n $, and $ 2\leqslant s \leqslant n $,
the Cremona group  $ {\rm Cr}_r $ of rank $ r $ contains the group $ {\rm Aut} (F_s) $, and
if $ n \geqslant 3 $,  the braid group $ B_n $ on $n$ strands.
\end{corollary}

\begin{proof}
Let $ G $ be the group $ {\rm SL}_2$. The variety
$\underline {G^n}$
is rational \cite[Cor. 14.14]{9-2}. Since $\dim G=3$, this shows that
$ {\rm Cr}_ {3n} $ contains
$ {\rm Aut}(\underline {G^{n}}) $. The statement now follows from
Corollary \ref{cor1} and the fact that $ {\rm Cr}_ {r + 1} $ contains $ {\rm Cr}_ {r} $
for any $ r $.
\end{proof}

\begin{remark} Denote by $b_n$ the minimal rank of the Cremona groups containing the braid group $B_n$ on $n$ strands.
In \cite{5-12a}, an embedding of the braid group $B_n$ into the group ${\rm GL}_{n(n-1)/2}$ is constructed, which implies the inequality
$b_n\leqslant n(n-1)/2$.  For $n\geqslant 8$, Corollary \ref{Cr} gives
a stronger upper bound: $$b_n\leqslant 3n.$$
\end{remark}

\end{subsec}

\begin{subsec}{\bf Groups nonembeddable into the Cremona groups.}  Since Co\-rollary \ref{Cr}
concerns subgroups of the Cremona group,
it is appropriate to complement it here with
the following remark on Cantat's question
about these subgroups.

\begin{remark}
In \cite {3-9}, are given examples of finitely-generated (and even finitely-presented) groups that do not admit embeddings into any Cremona group, which gives an answer to  Cantat's question about the existence of such groups (see also \cite{ 29-7}). These examples are based on Theorem 1.2 of \cite{4-10}, according to which the word problem is solvable in every finitely-generated subgroup of any Cremona group. However, the answer to the above question\,---\,and even in a stronger form, with the addition of the condition of simplicity of the subgroup\,---\,can be obtained
without using this theorem.

Namely, recall \cite{60-19} that an abstract group $H$ is called {\it Jordan} if
there exists a finite set $\mathcal F$ of finite groups such that every finite subgroup of $H$ is an extension of an Abelian group by a
group taken from $\mathcal F$.
According to \cite [p.\,188, Ex\-ample 6]{60-19}, Richard Thompson's group $ V $ is an example of a non-Jordan finitely presented group. Since any Cremona group is Jordan (see \cite[Cor.\,1.5]{1-1}, \cite{23}), the group $ V $ cannot be embedded in it. Moreover,
in addition to this property, $ V $ is simple, and therefore, any homomorphism of the group
$ V $ into a Cremona group is trivial (unlike \cite{4-10}, this proves Corollary 1.4
of \cite{4-10}
without usage of the
amplification of  the
Boone--Novikov construction  obtained in \cite{8-15}).
\end{remark}
\end{subsec}

\begin{subsec}{\bf Descending the action of \boldmath${\rm Aut}(F_n)$ to a quotient variety. } Let $ G $ be a connected (not necessarily affine) algebraic group.
Below we show that the variety $ X = \underline {G^{n}} $ is only one of the ``extreme'' cases in a family of algebraic varieties related to $ X $ and endowed with a natural action of the group
$ {\rm Aut}(F_n) $. For some of them (but not for all), this action is faithful, which gives
new examples of algebraic varieties,
the automorphism group of which contains $ {\rm Aut}(F_n) $, and therefore, by virtue of Proposition \ref{linea}, for $ n\geqslant 2 $, is nonamenable, and for $ n \geqslant 3 $, nonlinear and contains the braid group $ B_n $. Unlike $ X $,
they are no longer presented in the form of a Cartesian power of an algebraic variety,
not necessarily smooth and
endowed with a simply  transitive action of the group $ G^{n} $,
and among those of them that are affine, there are nonrational (and even not stably rational).

To obtain such varieties
one can, first, replace $ X $ with an appropriate $ {\rm Aut}(F_n) $-invariant open subset
$ U $ of $ X $ (in general, $U$ is not affine, even if $ G $ is affine). Such subsets do exist. For example, the set of points whose $ G $-orbits have maximal dimension with respect to the diagonal action of $ G $ on $ X $ by conjugation is open in $ X $, and from
Proposition \ref{pro}(g) it follows that it is ${\rm Aut}(F_n) $-invariant. If the action of $ {\rm Aut} (F_n) $ on $X $ is faithful, then the openness condition of $ U $ ensures  faithfulness of the action of $ {\rm Aut}(F_n) $ on $ U $.

Next, the transition from $ X $ to $ U $ can be complemented with the following construction.
Consider a closed subgroup $ S $ in $ G $ and its diagonal action on $ X $ by conjugation. Suppose that $ X $ contains an open subset $ U $ that is invariant under both
$ {\rm Aut}(F_n) $ and $S $ and admits a geometric or categorical quotient under the action of $ S $. Then, in view of Proposition \ref{pro}(g), the action of
$ {\rm Aut}(F_n) $ on $ U $ descends to the quotient variety. The results obtained below
show that under certain conditions
this action is faithful.

Below we explore two cases, in which $ U = G $. In the first, the group $ G $ is affine, and the group $ S $ is reductive, while in the second, no restrictions are imposed on the group $ G $, and the group $ S $ is finite. In the corresponding criteria for an element $ \sigma
\in {\rm Aut} (F_n) $ to belong to the kernel of the action of the group $ {\rm Aut}(F_n) $ on the quotient variety, we use the following notation:

If $w\in F_n$, $a\in G$, and $i\in\{1,\ldots, n\}$, we put
\begin{equation}\label{worddn}
X_{w, a, i}:=\{x=(g_1,\ldots, g_n)\in X\mid w(x)a^{-1}g_i^{-1}a=e\}.
\end{equation}
We have $(e,\ldots, e)\!\in\! X_{w, a, i}$,
therefore, $X_{w, a, i}\!\neq\! \varnothing$.
Being the fiber of the morphism
$$X\to G, \quad x=(g_1,\ldots, g_n)\mapsto w(x)a^{-1}g_i^{-1}a$$
over the point $e$, the set $X_{w, a, i}$
is closed in $X$.
\end{subsec}

\begin{subsec}{\bf Descending the action of \boldmath${\rm Aut}(F_n)$ to \boldmath${X}/\!\!/S$. }
Turning to the first case, we assume that $ G $ is a connected affine algebraic group, and $ S $ is its closed reductive subgroup.

Consider the diagonal action of the group $S$ on the affine algebraic variety $X=G^{n}$ by conjugation. Then the
$k$-algebra $k[{X}]^S$ is finitely generated, and if
 ${X}/\!\!/S$ is the affine algebraic variety with
  $k[{X}/\!\!/S]=k[{X}]^S$, and
$$\pi\colon {X}\to {X}/\!\!/S,$$
is the morphism determined by the identity
embedding $k[{X}]^S\hookrightarrow k[{X}]$, then the pair $(\pi, {X}/\!\!/S)$
is the categorical quotient for the action
of  $S$ on ${X}$ (see \cite{19-16}).
In view of Proposition \ref{pro}(g), any comorphism $\sigma_X^*$
preserves the al\-gebra  $k[X]^S$. Therefore, it indices an automorphism
$\sigma_{X/\!\!/S}\!\in\! {\rm Aut}(X/\!\!/S)$, for which
\begin{equation}\label{pssp}
\pi\circ\sigma_X=\sigma_{X/\!\!/S}\circ\pi.
\end{equation}

It arises  a homomorphism
\begin{equation}\label{action3}
{\rm Aut}(F_n)\to {\rm Aut}({X/\!\!/S}), \quad \sigma\mapsto (\sigma^{-1})_{X/\!\!/S}.
\end{equation}
Its kernel is described as follows:
\begin{lemma}\label{categ}
Let $G$ be a connected affine algebraic group, and let  $S$ be its closed reductive subgroup.
The following properties of an element
$\sigma\in {\rm Aut}(F_n)$ are equivalent:
\begin{enumerate}[\hskip 4.2mm\rm(a)]
\item
$\sigma$ lies in the kernel of homomorphism \eqref{action3}.
\item Every closed $S$-orbit lies in the set
\begin{equation}\label{formu}
\textstyle\bigcup_{s\in S}\Big(\bigcap_{i=1}^n X_{\sigma(f_i), s, i}\Big).
\end{equation}
\end{enumerate}
\end{lemma}

\begin{proof} The morphism $\pi$ is surjective, its fibers are $S$-inva\-riant, and for every point
$b\in X/\!\!/S$,
the fiber $\pi^{-1}(b)$
contains a unique closed  $S$-orbit $\mathcal O_b$
(see  \cite[\S2 and Append.\,1B]{19-16}).
It follows from \eqref{pssp} that the restriction of the morphism  $\sigma_X$ to the fiber $\pi^{-1}(b)$ is an $S$-equivariant isomorphism  between it and the fiber  $\pi^{-1}(\sigma_{X/\!\!/S}(b))$. In view of the uniqueness of closed orbits in the fibers, this means that
$\sigma_X(\mathcal O_b)=\mathcal O_{\sigma_{X/\!\!/S}(b)}$. Therefore, $\sigma_{X/\!\!/S}(b)=b$ if and only if  $\sigma_X(\mathcal O_b)=\mathcal O_b$.
In view of \eqref{mor} and definition \eqref{worddn},
this implies the equivalence of (a) and (b).
\end{proof}

\end{subsec}

\begin{subsec}{\bf Kernel of the action of \boldmath${\rm Aut}(F_n)$
on \boldmath${X}/\!\!/G$. } In the situation under consideration,
two extreme cases occur.

The first is the case of $ S = \{e\} $. It is considered in Theorem \ref{theo}, which shows that,
depending on $ G $ and $ S $,
 both possibilities are realized for homomorphism \eqref {action3}: in one, \eqref{action3} is an embedding, and in the other, it is not.

The second  is the case of $ S = G $ (so that $ G $ is reductive). It is considered in the following theorem.

\begin{theorem}\label{S=G} Let $G$ be a connected reductive algebraic group, and let  $S=G$.
\begin{enumerate}[\hskip 4.2mm\rm (a)]
\item If $n\geqslant 2$, then homomorphism \eqref{action3} is not an embedding.
\item If $n=1$, then the following properties are equivalent:
\begin{enumerate}[\hskip 0mm\rm (i)]
\item[$\rm(b_1)$] homomorphism \eqref{action3} is an embedding;
\item[$\rm(b_2)$]
the group $G$ contains a connected simple normal subgroup
of any of the following types:
\begin{equation}\label{type}
\mbox{${\sf A}_{\ell}$ with $\ell\geqslant 2$,\;
${\sf D}_\ell$ with odd $\ell$, \;
${\sf E}_6$.}
\end{equation}
\end{enumerate}
\end{enumerate}
\end{theorem}

\begin{proof}
Let $n>1$ and let $\sigma, \tau\in {\rm End}(F_n)$ be given by the formulas
\begin{equation}\label{conj}
\sigma(f_i)=
\begin{cases}f_1&\mbox{if $i=1$},\\
f_1f_if_1^{-1}&\mbox{if $i>1$,}\end{cases}
\quad
\tau(f_i)=
\begin{cases}f_1&\mbox{if $i=1$},\\
f_1^{-1}f_if_1&\mbox{if $i>1$.}\end{cases}
\end{equation}
Then $\sigma\circ\tau=\tau\circ\sigma=1$. Therefore, $\sigma\in {\rm Aut}(F_n)$,
and $\tau=\sigma^{-1}$.

Now, let $x\in X$ be a point \eqref{gh}. Then it follows from
\eqref{mor}, \eqref{conj} that
$$\sigma_X(x)=(g_1, g_1g_2g_1^{-1},\ldots, g_1g_ng_1^{-1})=g_1\cdot x,$$
so $x$ and $\sigma_X(x)$ lie in the same  $S$-orbit.
Therefore, the point  $\pi(x)$
is fixed with respect to
 $\sigma_{X/\!\!/S}$. Since $\pi$ is surjective, this shows that
$\sigma$ lies in the kernel of homomorphism \eqref{action3}.
In view of $\sigma\neq 1$,
this proves\;(a).

Now, let $n=1$, so that $X=G$. Then ${\rm Aut}(F_n)$
is the group of order two, and if
 $\sigma\in {\rm Aut}(F_n)$, $\sigma\neq 1$, then $\sigma(f_1)=f_1^{-1}$, so that $\sigma_X(g)=g^{-1}$ for any $g\in G$. Every fiber of the morphism $\pi$
 contains a unique orbit consisting of semisimple elements, and it is
 the unique closed orbit in this fiber
 (see \cite{20-23}).

From this and Lemma\;\ref{categ}
the equivalence of the following properties follows:
\begin{enumerate}[\hskip 2.0mm\rm(i)]
\item $\sigma$ lies in the kernel of homomorphism  \eqref{action3};
\item $g$ and $g^{-1}$ are conjugate for any semisimple element
 $g\in G$.
\end{enumerate}

Since the intersection of any semisimple
conjugacy class with a fixed maximal torus of the group $ G $ is nonempty
 (see \cite[Thm. 11.10]{9-2}) and is an orbit of the normalizer of this torus
 (see \cite[6.1]{20-23} or \cite[1.1.1]{23-22}),  property (ii)
 is equivalent to the fact that the Weyl group $W$ of the group $ G $, considered as the subgroup of the group ${\rm GL}({\rm Lie}(T))$,
 contains $ -1 $.
 This, in turn, is equivalent to the fact that $ -1 $ is contained in the Weyl group of every nontrivial connected simple normal subgroup of $ G $. Let $C$ be a Weyl chamber in
 ${\rm Lie}(T)$. Since $-C$ is a Weyl chamber as well, the simple transitivity of the action of $W$ on the set of Weyl chambers implies the existence of a unique element $w_0\in W$ such that $w_0(C)=-C$. In view of $(-1)(C)=-C$, this means that the inclusion
 $-1\in W$ is equivalent to the equality $w_0=-1$. In \cite[Table I--IX]{21-3}, an explicit description of the element $w_0$ is given for each connected simple
algebraic group. It follows from it that the equality $w_0= -1$ for such a group is equivalent to the fact that
the type of this group
is not contained in list \eqref{type}.
This completes the proof.
\end{proof}
\end{subsec}

\begin{subsec}{\bf Kernel of the action of \boldmath${\rm Aut}(F_n)$
on \boldmath${X}/\!\!/S$ for finite $S$. }  Now, we consider the second case, where $ G $ is any (not necessarily affine) connected algebraic group, and $ S $ is its finite subgroup.

According to \cite[Prop.\;19 and Example 2) on p. 50]{22-21}, in this case, there exist an algebraic variety $ X/S $ and a morphism
\begin{equation} \label{gequ}
\rho \colon X \to X/S
\end{equation}
such that the pair $(\rho, X/S)$ is the geometric quotient for the diagonal action of $S$ on $X$ by conjugation.
In particular,
\begin{equation}\label{fields}
\rho^*(k(X/S))=k(X)^S.
\end{equation}

For every element $ \sigma \in { \rm Aut}(F_n) $, from Proposition \ref{pro}(g) and the properties of geometric quotient it follows the existence of
an element $ \sigma_{X/S} \in {\rm Aut}(X/S) $ such that
\begin{equation*}
\rho\circ\sigma_X=\sigma_{X/S}\circ\rho.
\end{equation*}

It arises a homomorphism
\begin{equation}\label{gqh}
{\rm Aut}(F_n)\to {\rm Aut}(X/S),\quad \sigma\mapsto (\sigma^{-1})_{X/S}.
\end{equation}
Its kernel is described as follows:

\begin{lemma}\label{lemma} Let
$G$ be a connected \textup(not necessarily affine\textup) algebraic group and let $S$ be its finite subgroup.
The following properties of an element
$\sigma\in {\rm Aut}(F_n)$ are equivalent:
\begin{enumerate}[\hskip 4.2mm\rm(a)]
\item $\sigma$ lies in the kernel of homomorphism \eqref{gqh}.
\item There exists an element $s\in S$ such that
\begin{equation}\label{===}
X=X_{\sigma(f_1), s, 1}=\ldots=X_{\sigma(f_n), s, n}.
\end{equation}
\end{enumerate}
If the group $G$ is nonsolvable
and
         $({\rm b})$ holds, then the following con\-di\-tions are equivalent:
\begin{enumerate}[\hskip 4.2mm\rm(i)]
\item[$({\rm b}_1)$] $\sigma=1$;
\item[$({\rm b}_2)$] $s\in {\mathscr C}(G)$.
\end{enumerate}
\end{lemma}
\begin{proof}
Since \eqref{gequ} is a geometric factor, the morphism
$\rho$ is surjective and each of its fibers is an $S$-orbit. An argument similar to that used in the proof of  Lemma \ref{categ} shows that
condition (a) is equivalent to the condition that each $S$-orbit is contained in the set \eqref{formu}, i.e.,
the condition that the set \eqref{formu} coincides with the entire set $X$.
Since the algebraic
variety $X$ is irreducible, and the set \eqref{formu} is a union of finitely many
(because $S$ is finite)
closed sets of the form $\textstyle\bigcap_{i=1}^n X_{\sigma(f_i), s, i }$, the variety $X$ must coincide with one of them. Finally, the equality $X=\textstyle\bigcap_{i=1}^n X_{\sigma(f_i), s, i}$ is obviously equivalent to the system of equalities \eqref{===}. This proves that (a) and (b) are equivalent.

Now, let the group $G$ be unsolvable
         and let $({\rm b})$ holds. Consider an arbitrary point $x=(g_1,\ldots, g_n)\in X$. If
$({\rm b}_1)$ holds, then for each $i\in\{1,\ldots, n\}$, we have the equality $\sigma(f_i)=f_i$, and therefore, in view of  \eqref{worddn}
and \eqref{===}, the equality
$g_i=sg_is^{-1}$. Since $g_i$
may be any element of $G$, this means that $({\rm b}_2)$ holds.
Conversely, if $({\rm b}_2)$ holds, then for any $i$, it follows from \eqref{worddn} and \eqref{===} that
$\sigma(f_i)(x)=g_i$, i.e., $\sigma$ lies in the kernel of homomorphism \eqref{restri}, which, in view of   Theorem \ref{theo}, is trivial. Hence $({\rm b}_1)$ holds. This proves that $({\rm b}_1)$ and $({\rm b}_2)$ are equivalent.
\end{proof}
\end{subsec}

\begin{subsec}{\bf The case of unsolvable \boldmath $G$ and finite $S$.}\label{Sfinite} Now, we explore as to when in the considered situation homomor\-phism \eqref{gqh} is an embedding.
Since the action of the group $ S $ on $ G $ by conjuga\-tion is trivial if and only if
$ S \subseteq {\mathscr C}(G) $, in what follows we can (and shall) assume that the group $ S $ does not lie in $ {\mathscr C}(G ) $ (i.e., is noncentral).

\begin{theorem}\label{fg}
Let $G$ be a nonsolvable connected \textup(not necessarily affine\textup) algebraic group, and let
 $S$ be its noncentral finite subgroup.
Then for the diagonal action of $S$ on $X:=G^{n}$ by conjugation,
the following hold:
\begin{enumerate}[\hskip 4.2mm \rm(a)]
\item Homomorphism
\eqref{gqh} is an embedding of ${\rm Aut}(F_n)$ into
${\rm Aut}(X/S)$.
\item
The group ${\rm Aut}(X/S)$ is
\begin{enumerate}[\hskip 0mm $\rm(b_1)$]
\item
nonamenable if $n\geqslant 2$,
\item
non\-li\-near and contains the braid group
 $B_n$ if $n\geqslant 3$.
 \end{enumerate}
 \end{enumerate}
\end{theorem}

\begin{proof}
Arguing by contradiction, assume that the kernel of homomor\-phism \eqref{gqh} contains an element $ \sigma \in {\rm Aut} (F_n) $, $ \sigma \neq 1 $.
Then by Lemma \ref{lemma}, there is an element
\begin{equation}\label{eles}
s\in S\setminus {\mathscr C}(G),
\end{equation}
such that equalities  \eqref{===} hold.  In view of \eqref{worddn}, this means that for each  $i\in \{1,\ldots, n\}$, the following group identity holds:
\begin{equation}\label{condi1}
\sigma(f_i)(g_1,\ldots, g_n)=sg_is^{-1}\quad \mbox{for all  $g_1,\ldots, g_n\in G$.}
\end{equation}
In particular, for every
$g\in G$, the equality obtained by taking
$g_1=\ldots=g_n=g$ in \eqref{condi1} holds.
Since $ \sigma (f_i) $
has form \eqref{aword}, this means
the existence of an integer $ d $ such that
the following group identity holds:
\begin{equation}\label{sgs1}
g^d=sgs^{-1}\quad\mbox{for each $g\in G$.}
\end{equation}

Notice that
\begin{equation}\label{neq}
d\neq 1\quad \mbox{and}\quad d\neq -1.
\end{equation}
Indeed, in view of \eqref{sgs1}, if $ d = 1 $, then $ s \in {\mathscr C}(G) $, which contradicts \eqref{eles}.
If $d=-1$, then for any
$g, h\in G$ the following equality holds
\begin{equation*}
h^{-1}g^{-1}=(gh)^{-1}\overset{\eqref{sgs1}}{=}s(gh)s^{-1}
=sgs^{-1}shs^{-1}\overset{\eqref{sgs1}}{=}g^{-1}h^{-1},
\end{equation*}
which means commutativity of the group $G$ and contradicts its non\-solv\-ability.

Next, if $ r $ is a positive integer, then the following group identity holds:
\begin{equation}\label{idenex}
s^rgs^{-r}=g^{d^r}\quad \mbox{for each $g\in G$.}
\end{equation}
Indeed, \eqref{idenex} becomes \eqref{sgs1} for $ r = 1 $. Arguing by induction, from
$ s^{r-1} gs^{- r + 1} = g^{d^{r-1}} $ we obtain
\begin{equation*}\label{sgs2}
s^rgs^{-r}{=}s(s^{r-1}gs^{-r+1})s^{-1}=
sg^{d^{r-1}}s^{-1}
\overset{\eqref{sgs1}}{=}(g^{d^{r-1}})^d=g^{d^r},
\end{equation*}
as stated.

Since the group $S$ is finite, the order of the element
$s$ in group identity \eqref{sgs1} is finite. Let $r$ in \eqref{idenex} be equal to this order. Then \eqref{idenex} becomes the group identity
\begin{equation}\label{idenex1}
e=g^{d^r-1}\quad \mbox{for each $g\in G$.}
\end{equation}

Since, in view of \eqref{neq}, we have $d^r-1\neq 0$,
it follows from group identity  \eqref{idenex1} that
  $G$ is a torsion group whose element orders are upper bounded.

As in the proof of the Theorem \ref{theo}
(whose notation we retain), the affine algebraic group $G_{\rm aff}$ is unsolvable. Hence
it contains a nontrivial semisimple element, and therefore a torus of positive dimension (see \cite[Thms. 4.4, 11.10]{9-2}).

But
the set of orders of elements of the torsion subgroup of any torus of positive dimension is not upper bounded (see \cite[Prop.\,8.9]{9-2}).
The resulting contradiction completes the proof of (a).
Statement (b) follows from (a) and
Proposition \ref{linea}.
\end{proof}
\end{subsec}

\begin{subsec}{\bf Nonrational varieties \boldmath $X/\!\!/S$.}
 The supplied by Theorem \ref{fg}  family of algebraic varieties $ X/S $ such that $ {\rm Aut} (F_n) $ lies in ${\rm Aut}(X/S)$,
 contains rational affine algebraic varieties:
 such is $ X $ itself if $ G $ is unsolvable and affine, because the underlying
variety of any connected affine algebraic group is rational (see \cite[Cor. 14.14]{9-2}).

We will now show that this family contains
also affine nonrational (and even not stably rational) algebraic varieties.

We use the following known statement (see, e.g., \cite[Thm.\,1]{16-18}).

\begin{lemma}\label{strat}
If the field of invariant rational functions of some faithful linear action of a finite group on a finite-dimensional vector space over $ k $ is stably rational over $ k $, then the same property holds for any other such action of this group.
\end{lemma}

Let $p$ be a prime integer other than ${\rm char}(k)$.
In \cite{17-20} are found finite groups $F$ of order $p^9$ and group embeddings
\begin{equation*}
\iota\colon F\hookrightarrow {\rm GL}(V),
\end{equation*}
where $V$ is a finite-dimensional vector space over $k$,
such that the field of $\iota(F)$-invariant rational functions on $V$ is not stably rational over
 $k$. In view of Lemma \ref{strat}, replacing $V$ and $\iota$ if needed, we
 can (and shall) assume that
\begin{equation}\label{ncent}
\iota(F)\cap {\mathscr C}\big({\rm GL}(V)\big)=\{{\rm id}_V\}.
\end{equation}
Indeed, let
$L$ be a one-dimensional vector space over $k$. Since
$${\mathscr C}\big({\rm GL}(V\oplus L)\big)=\{c\cdot {\rm id}_{V\oplus L}\mid c\in k, c\neq 0\},$$ for the group embedding
\begin{equation*}
\iota'\colon F\hookrightarrow {\rm GL}(V\oplus L),\quad f\mapsto \iota(f)\oplus{\rm id}_L,
\end{equation*}
we have $\iota'(F)\cap {\mathscr C}\big({\rm GL}(V\oplus L)\big)=\{{\rm id}_{V\oplus L}\}$.

Now we put
\begin{equation*}
G:={\rm GL(V)},\quad S:=\iota(F).
\end{equation*}
It follows from \eqref{ncent} that the diagonal action of the group $S$
on the vector space ${\rm End(V)}^{\oplus n}$ by conjugation
is a faithful linear action. There\-fore, in view of Lemma  \ref{strat},
the field of rational $S$-invariant functions on
${\rm End(V)}^{\oplus n}$
is not stably rational over
 $k$. Since  $X:=G^{n}$ is an
$S$-invariant open subset of ${\rm End}(V)^{\oplus n}$, this implies that
the field $k(X)^S$ for the diagonal action of  $S$ on $X$ by conjugation
is not stably rational over $k$. This and \eqref{fields} yield that  the algebraic variety
$X/S$ is not stably rational. Since the group $G$ is affine,
we have $X/S=X/\!\!/S$ (see \cite[Prop. 18 on p. 48]{22-21}), so that the algebraic variety $X/S$  is affine. Finally,  in view of nonsolvability of the group $G$,
it follows from Theorem \ref{fg} that
the group ${\rm Aut} (X/S)$
contains ${\rm Aut}(F_n)$.

\begin{remark}
Currently, the groups of orders $ p^6 $ and $ p^5 $ are known whose fields of rational invariants of faithful linear actions are not stably rational (see details and references in
\cite[Rem. on p. 414]{16-18}). They can be taken as $ F $ in the construction described in this section.
\end{remark}
\end{subsec}



\begin{thebibliography}{10}

\bibitem{1-1} C. Birkar,
{\it Singularities of linear systems and boundedness of Fano varieties},
{\tt {\footnotesize arXiv}:1609.05543} (2016).

\bibitem{2} A. Borel, {\it On free subgroups of semi-simple groups}, Enseign. Math. (2), {\bf 29}
(1983), no. 1-2, 151--164.

\bibitem{9-2}
A. Borel, {\it Linear Algebraic Groups}, Graduate Texts in Mathematics, Vol. 126, Sprin\-ger-Verlag, New York, 1991.

\bibitem{21-3} N.\;Bourbaki, {\it Groupes et Alg\`ebres de Lie}, Chaps. IV, V, VI, Hermann, Paris, 1968.


\bibitem{10-4} E. Breuillard, B. Green, R. Guralnick, T. Tao, {\it Strongly dense free subgroups of semi\-simple algebraic groups}, Israel J. Math. {\bf  192}
    (2012), 347--379.

\bibitem{13-5}    E. Breuillard, B. Green, R. Guralnick, T. Tao,
{\it Expansion in finite simple groups of Lie type},
J. Eur. Math. Soc. {\bf 17} (2015), 1367--1434.


    \bibitem{15-6}  S. Cantat, {\it Linear groups in Cremona groups}, Int. conf. {\it Birational and Affine Geometry},
April 25, 2012, Moscow, Steklov Math. Inst.\;RAS,
{\tt {\footnotesize http:$/\!\!/$www.\break mathnet.ru/php/presentation.phtml?presentid=4848$\&$option$\underline{\ }$lang=eng}}

\bibitem{29-7} S. Cantat, {\it A remark on groups of birational transformations and the word problem \textup(af\-ter Yves de Cornulier\textup)} (2013),
    {\tt {\footnotesize https:$/\!\!/$www.semanticscholar.\break org/paper/A-REMARK-ON-GROUPS-OF-BIRATIONAL-TRANSFORMATIONS-DE-\break Cantat/8cd955de23876f349a58d43cc135bfd7d3da5d95}}.

    \bibitem{14-8} D. Cerveau, J. D\'eserti, {\it Transformations Birationnelles de Petit Degr\'e}, Cours Sp\'ecialis\'es, Vol. 19, Collection SMF, 2013.

\bibitem{3-9} Y. Cornulier, {\it Nonlinearity of some subgroups of the planar Cremona group},
{\tt {\footnotesize arXiv}:1701.00275v1} (2017).

\bibitem{4-10} Y. Cornulier, {\it Sofic profile and computability
of Cremona groups}, Michigan Math. J. {\bf 62} (2013), 823--841.


\bibitem{12-11} D. B. A. Epstein, {\it Almost all subgroups of a Lie group are free}, J. Algebra
{\bf 19} (1971), 261--262.


 \bibitem{5-12} E. Formanek, C. Procesi, {\it The automorphism group of a free group is not linear}, J. Algebra {\bf 149} (1992), 494--499.

     \bibitem{5-12a} D. Krammer, {\it Braid groups are linear}, Annals of  Math. {\bf 155} (2002), 131--156.

     \bibitem{24-13} R.\;C.\;Lyndon, P.\;E.\;Schupp, {\it Combinatorial Group Theory},
     Ergebnisse der Mathematik und ihrer Grenzgebiete, Bd. 89, Springer-Verlag, Berlin, 1977.

 \bibitem{18-14} W. Magnus, A. Karrass, D. Solitar, {\it Combinatorial Group Theory}, Interscience, New York, 1966.

  \bibitem{8-15}   C. F. Miller III, {\it The word problem in quotients of a group}, in; {\it Aspects of Effective
Algebra} (Clayton, 1979), Upside Down,Yarra Glen, Victoria, 1981, pp. 246–250.

  \bibitem{19-16} D. Mumford, J. Fogarty, {\it Geometric Invariant Theory}, Second Edition,
  Er\-geb\-nisse der Mathematik und ihrer Grenzgebiete, Bd. 34, Springer-Verlag, Berlin, 1982.

 \bibitem{7-17}  V.\;L.\;Popov,\;{\it Some subgroups of the Cremona groups}, in: {\it Affine Algebraic Geometry}, Pro\-ceedings of the conference on the occasion of M. Miyanishi's 70th birthday (Osaka, Japan, 3--6 March 2011), World Scientific Publ., Singapore, 2013, pp.\;213--242.

 \bibitem{16-18}     V.\;L.\;Popov, {\it Rationality and the FML invariant}, J. Ramanujan Math. Soc. {\bf 28A} (2013), 409--415.

     \bibitem{60-19}  V.\;L.\;Popov,\;{\it Jordan groups and auto\-morphism groups of algebraic varieties}, in:    {\it Auto\-morphisms
in Birational and Affine Geometry} (Levico Terme, 2012), Springer Proc.
Math. Stat., Vol. 79, Springer, Cham, 2014, pp.\;185--213.

\bibitem{22}  V.\;L.\;Popov,\;{\it Embeddings of groups ${\rm Aut}(F_n)$ into automorphism groups of algebraic varieties}, {\tt arXiv:2106.02072v1} (2021).

    \bibitem{23}    Yu. Prokhorov, C. Shramov, {\it Jordan property for Cremona groups},
    Amer. J. Math. {\bf 138} (2016), 403--418.

     \bibitem{17-20} D.\;J.\;Saltman, {\it Noether's problem over an algebraically closed field}, Invent. Math. {\bf 77} (1984), 71--84.

         \bibitem{22-21} J.-P.\;Serre, {\it Algebraic Groups and Class Fields}, Graduate Texts in Mathe\-matics, Vol. 117, Springer, New York, 1997.

          \bibitem{23-22}  J.-P.\;Serre, {\it Sous-groupes finis des groupes de Lie}, S\'eminaire Bourbaki, 1998--99, exp. ${\rm n}^\circ\,864$.

          \bibitem{20-23} R.\;Steinberg, {\it Regular elements in semisimple algebraic groups}, Publ. Math. l'IHES {\bf 25} (1965), 49--80.



\end{thebibliography}
\end{document}